\documentclass[12pt,twoside]{amsart}
\usepackage{amsmath, amsthm, amscd, amsfonts, amssymb, graphicx}
\usepackage[bookmarksnumbered, plainpages]{hyperref}

\newtheorem{thm}{Theorem}[section]
\newtheorem{cor}[thm]{Corollary}
\newtheorem{lem}[thm]{Lemma}

\newtheorem{defn}[thm]{Definition}

\numberwithin{equation}{section}

\begin{document}

\title{\bf  Anomaly cancellation formulas and $E_8$ bundles}
\author{Yong Wang, Yuchen Yang}

\thanks{{\scriptsize
\hskip -0.4 true cm \textit{2010 Mathematics Subject Classification:}
58C20; 57R20; 53C80.
 }}

\maketitle

\begin{abstract}
 Using $E_8$ bundles, we construct some modular forms over $SL(2,{\bf Z})$, $\Gamma^0(2)$ and $\Gamma_0(2)$. By these modular forms, we get
 some new anomaly cancellation formulas of characteristic forms.
\end{abstract}

\vskip 0.2 true cm


\pagestyle{myheadings}
\markboth{\rightline {\scriptsize Yong Wang and Yuchen Yang}}
         {\leftline{\scriptsize Anomaly cancellation formulas and $E_8$ bundles}}

\bigskip
\bigskip


\section{ Introduction}
\quad In 1983, the physicists Alvarez-Gaum\'{e} and Witten \cite{AW}
  discovered the "miraculous cancellation" formula for gravitational
  anomaly which reveals a beautiful relation between the top
  components of the Hirzebruch $\widehat{L}$-form and
  $\widehat{A}$-form of a $12$-dimensional smooth Riemannian
  manifold. Kefeng Liu \cite{Li1} established higher dimensional "miraculous cancellation"
  formulas for $(8k+4)$-dimensional Riemannian manifolds by
  developing modular invariance properties of characteristic forms.
  These formulas could be used to deduce some divisibility results. In
  \cite{HZ1}, \cite{HZ2}, \cite{CH}, some more general cancellation formulas that involve a
  complex line bundle and their applications were established. In \cite{HLZ1}, using the Eisenstein series, a more general cancellation
  formula was derived.  In \cite{HLZ2}, Han, Liu and Zhang showed that both of the Green-Schwarz anomaly factorization formula
for the gauge group $E_8\times E_8$ and the Horava-Witten anomaly factorization formula for the gauge
group $E_8$ could be derived through modular forms of weight $14$. This answered a question of J.
H. Schwarz. They also established generalizations of these factorization formulas and obtaind a new
Horava-Witten type factorization formula on $12$-dimensional manifolds. In \cite{HHLZ}, Han, Huang, Liu and Zhang introduced a modular form of weight $14$ over $SL(2,{\bf Z})$ and a modular form of weight $10$ over $SL(2,{\bf Z})$ by $E_8$ bundles and they got some interesting
anomaly cancellation formulas on $12$-dimensional manifolds.
  In \cite{CHZ}, Chen, Han and Zhang defined an integral modular form of weight $2k$ for a $4k+2$-dimensional $spin^c$ manifold.
 In this paper, we twist the Chen-Han-Zhang $SL(2,{\bf Z})$ modular form by $E_8$ bundles and get $SL(2,{\bf Z})$ modular forms of weight $14$ and $10$ for $14$ and $10$-dimensional $spin^c$ manifolds. In \cite{Li1}, Liu introduced some $\Gamma^0(2)$ and $\Gamma_0(2)$ modular forms and got
 some interesting anomaly cancellation formulas. In this paper, we also twist the Liu's modular form by $E_8$ bundles and get $\Gamma^0(2)$ and $\Gamma_0(2)$ modular forms of weight $14$ and $10$ for a $12$-dimensional spin manifold. By these modular forms, we get
 some new anomaly cancellation formulas of characteristic forms.\\
  \indent This paper is organized as follows: In Section 2, we twist the Chen-Han-Zhang $SL(2,{\bf Z})$ modular form by $E_8$ bundles and get $SL(2,{\bf Z})$ modular forms of weight $14$ and $10$ for $14$ and $10$-dimensional $spin^c$ manifolds and we get
 some anomaly cancellation formulas of characteristic forms.
In Section 3, we twist the Liu's modular form by $E_8$ bundles and get $\Gamma^0(2)$ and $\Gamma_0(2)$ modular forms of weight $14$ and $10$ for a $12$-dimensional spin manifold. By these modular forms, we get
 some new anomaly cancellation formulas of characteristic forms.

  \vskip 1 true cm

\section{Anomaly cancellation formulas for $14$ and $10$-dimensional $spin^c$ manifolds}

\indent For the basic representation theory for the affine $E_8$, we can see Section 1 in \cite{HLZ2}. Here we omit it. For the principle $E_8$ bundle $P_i$, $i=1,2$, consider the associated bundles $\mathcal{V}_i=\sum_{k=0}^{\infty}(P_i\times_{\rho_k}V_k)q^k\in K(M)[[q]].$ We let $W_i=P_i\times_{\rho_1}V_1$. Let $M$ be a $14$-dimensional spinc manifold and $L$ be the complex line bundle associated to the given spinc structure on $M$. We also
consider $L$ as a real vector bundle denoted by $L_R$. \\
\indent Denote by $c=c_1(L)=2\pi\sqrt{-1}u$ the first Chern class of $L$. Let
$\varphi(\tau)=\prod_{n=1}^{\infty}(1-q^n)$, where $q=e^{2\pi\sqrt{-1}\tau}$ and $\tau\in H$ the upper half plane.
We recall the four Jacobi theta functions are
   defined as follows( cf. \cite{Ch}):
 \begin{equation}  \theta(v,\tau)=2q^{\frac{1}{8}}{\rm sin}(\pi
   v)\prod_{j=1}^{\infty}[(1-q^j)(1-e^{2\pi\sqrt{-1}v}q^j)(1-e^{-2\pi\sqrt{-1}v}q^j)],
   \end{equation}
\begin{equation}\theta_1(v,\tau)=2q^{\frac{1}{8}}{\rm cos}(\pi
   v)\prod_{j=1}^{\infty}[(1-q^j)(1+e^{2\pi\sqrt{-1}v}q^j)(1+e^{-2\pi\sqrt{-1}v}q^j)],\end{equation}
\begin{equation}\theta_2(v,\tau)=\prod_{j=1}^{\infty}[(1-q^j)(1-e^{2\pi\sqrt{-1}v}q^{j-\frac{1}{2}})
(1-e^{-2\pi\sqrt{-1}v}q^{j-\frac{1}{2}})],\end{equation}
\begin{equation}\theta_3(v,\tau)=\prod_{j=1}^{\infty}[(1-q^j)(1+e^{2\pi\sqrt{-1}v}q^{j-\frac{1}{2}})
(1+e^{-2\pi\sqrt{-1}v}q^{j-\frac{1}{2}})],\end{equation}
By (2.3)-(2.7) in \cite{HLZ2}, we have that there are formal two forms $y_l^i,1\leq l\leq 8,i=1,2$ such that
\begin{equation}
\varphi(\tau)^8{\rm ch}(\mathcal{V}_i)=\frac{1}{2}\left(\prod_{l=1}^8\theta_1(y_l^i,\tau)+\prod_{l=1}^8\theta_2(y_l^i,\tau)+\prod_{l=1}^8\theta_3(y_l^i,\tau)\right),
\end{equation}
and
\begin{equation}
\sum_{l=1}^8(2\pi\sqrt{-1}y_l^i)^2=-\frac{1}{30}c_2(W_i),
\end{equation}
where $c_2(W_i)$ denotes the second Chern class of $W_i$. Let $E_2(\tau)$ is the Eisenstein series satisfying:
\begin{equation}
E_2(\tau+1)=E_2(\tau),~~E_2(-\frac{1}{\tau})=\tau^2E_2(\tau)-\frac{6\sqrt{-1}\tau}{\pi},
\end{equation}
and
\begin{equation}
E_2(\tau)=1-24q-72q^2+O(q^3).
\end{equation}
Let
$A:=p_1(M)-p_1(L_R)+\frac{1}{30}(c_2(W_i)+c_2(W_j)),$ where $p_1(M),p_1(L_R)$ denote the first Pontryajin classes of $M$ and $L_R$.
Let \begin{align}
Q(M,P_i,P_j,\tau)=&\left\{e^{\frac{1}{24}E_2(\tau)A}\widehat{A}(TX){\rm exp}(\frac{c}{2}){\rm ch}\left[\bigotimes _{n=1}^{\infty}S_{q^n}(\widetilde{T_CM})
\otimes
\bigotimes _{m=1}^{\infty}\wedge_{-q^m}(\widetilde{L_C})\right]\right.\\\notag
&\left.\cdot\varphi(\tau)^{16}{\rm ch}(\mathcal{V}_i){\rm ch}(\mathcal{V}_j)\right\}^{(14)}.
\end{align}
Then
\begin{align}&Q(M,P_i,P_j,\tau)=\left\{e^{\frac{1}{24}E_2(\tau)A}\left(\prod_{j=1}^{7}\frac{x_j\theta'(0,\tau)}{\theta(x_j,\tau)}\right)
\frac{\sqrt{-1}\theta(u,\tau)}{\theta_1(0,\tau)\theta_2(0,\tau)
\theta_3(0,\tau)}\right.\\\notag
&\left.\frac{1}{4}\left(\prod_{l=1}^8\theta_1(y_l^i,\tau)+\prod_{l=1}^8\theta_2(y_l^i,\tau)+\prod_{l=1}^8\theta_3(y_l^i,\tau)\right)
\left(\prod_{l=1}^8\theta_1(y_l^j,\tau)+\prod_{l=1}^8\theta_2(y_l^j,\tau)+\prod_{l=1}^8\theta_3(y_l^j,\tau)\right)
\right\}^{(14)},
\end{align}
where $\pm 2\pi\sqrt{-1}x_j,1\leq j\leq 7$ denote the Chern roots of $T_CM$.
 One
has the following transformation laws of theta functions (cf. \cite{Ch} ):
\begin{equation}\theta(v,\tau+1)=e^{\frac{\pi\sqrt{-1}}{4}}\theta(v,\tau),~~\theta(v,-\frac{1}{\tau})
=\frac{1}{\sqrt{-1}}\left(\frac{\tau}{\sqrt{-1}}\right)^{\frac{1}{2}}e^{\pi\sqrt{-1}\tau
v^2}\theta(\tau v,\tau);\end{equation}
\begin{equation}\theta_1(v,\tau+1)=e^{\frac{\pi\sqrt{-1}}{4}}\theta_1(v,\tau),~~\theta_1(v,-\frac{1}{\tau})
=\left(\frac{\tau}{\sqrt{-1}}\right)^{\frac{1}{2}}e^{\pi\sqrt{-1}\tau
v^2}\theta_2(\tau v,\tau);\end{equation}
\begin{equation}\theta_2(v,\tau+1)=\theta_3(v,\tau),~~\theta_2(v,-\frac{1}{\tau})
=\left(\frac{\tau}{\sqrt{-1}}\right)^{\frac{1}{2}}e^{\pi\sqrt{-1}\tau
v^2}\theta_1(\tau v,\tau);\end{equation}
\begin{equation}\theta_3(v,\tau+1)=\theta_2(v,\tau),~~\theta_3(v,-\frac{1}{\tau})
=\left(\frac{\tau}{\sqrt{-1}}\right)^{\frac{1}{2}}e^{\pi\sqrt{-1}\tau
v^2}\theta_3(\tau v,\tau),\end{equation}
 \begin{equation}\theta'(v,\tau+1)=e^{\frac{\pi\sqrt{-1}}{4}}\theta'(v,\tau),~~
 \theta'(0,-\frac{1}{\tau})=\frac{1}{\sqrt{-1}}\left(\frac{\tau}{\sqrt{-1}}\right)^{\frac{1}{2}}
\tau\theta'(0,\tau).\end{equation}
\begin{defn} A modular form over $\Gamma$, a
 subgroup of $SL_2({\bf Z})$, is a holomorphic function $f(\tau)$ on
 $\textbf{H}$ such that
\begin{equation} f(g\tau):=f\left(\frac{a\tau+b}{c\tau+d}\right)=\chi(g)(c\tau+d)^kf(\tau),
 ~~\forall g=\left(\begin{array}{cc}
\ a & b  \\
 c & d
\end{array}\right)\in\Gamma,\end{equation}
\noindent where $\chi:\Gamma\rightarrow {\bf C}^{\star}$ is a
character of $\Gamma$. $k$ is called the weight of $f$.
\end{defn}
By (2.7) and (2.11)-(2.15), we have
\begin{lem}
$Q(M,P_i,P_j,\tau)$ is a modular form over $SL_2({\bf Z})$ with the weight $14$.
\end{lem}
By
\begin{equation}
{\rm ch}(\mathcal{V}_i)=1+(248-c_2(W_i)+\cdots)q+\cdots,
\end{equation}
and (2.8), we have by expanding the $q$-series:
\begin{align}
&e^{\frac{1}{24}E_2(\tau)A}\widehat{A}(TX){\rm exp}(\frac{c}{2}){\rm ch}\left[\bigotimes _{n=1}^{\infty}S_{q^n}(\widetilde{T_CM})
\otimes
\bigotimes _{m=1}^{\infty}\wedge_{-q^m}(\widetilde{L_C})\right]\varphi(\tau)^{16}{\rm ch}(\mathcal{V}_i){\rm ch}(\mathcal{V}_j)\\\notag
&=(e^{\frac{1}{24}A}-e^{\frac{1}{24}A}Aq+O(q^2))\widehat{A}(TX){\rm exp}(\frac{c}{2})
{\rm ch}\left[(1+q\widetilde{T_CM}+O(q^2))(1-q\widetilde{L_C}+O(q^2))\right]\\\notag
&\cdot(1-16q+O(q^2))(1+{\rm ch}({W_i})q+O(q^2))(1+{\rm ch}({W_j})q+O(q^2))\\\notag
&=e^{\frac{1}{24}A}\widehat{A}(TX){\rm exp}(\frac{c}{2})+\left[e^{\frac{1}{24}A}\widehat{A}(TX){\rm exp}(\frac{c}{2})
{\rm ch}(\widetilde{T_CM}-\widetilde{L_C}-16+W_i+W_j)\right.\\\notag
&\left.-e^{\frac{1}{24}A}A\widehat{A}(TX){\rm exp}(\frac{c}{2})\right]q+
O(q^2).\notag
\end{align}
It is well known that modular forms over $SL_2({\bf Z})$ can be expressed as polynomials of the Einsentein series $E_4(\tau)$ and $E_6(\tau)$,
where
 \begin{equation}
E_4(\tau)=1+240q+2160q^2+6720q^3+\cdots,
\end{equation}
\begin{equation}
E_6(\tau)=1-504q-16632q^2-122976q^3+\cdots.
\end{equation}
Their weights are $4$ and $6$ respectively. Since $Q(M,P_i,P_j,\tau)$ is a modular form over $SL_2({\bf Z})$ with the weight $14$, it must be
a multiple of
\begin{equation}
E_4(\tau)^2E_6(\tau)=1-24q-196632q^2+\cdots.
\end{equation}
So
\begin{align}
&\left\{e^{\frac{1}{24}A}\widehat{A}(TX){\rm exp}(\frac{c}{2})
{\rm ch}(\widetilde{T_CM}-\widetilde{L_C}-16+W_i+W_j)-e^{\frac{1}{24}A}A\widehat{A}(TX){\rm exp}(\frac{c}{2})\right\}^{(14)}\\\notag
&=-24\left\{e^{\frac{1}{24}A}\widehat{A}(TX){\rm exp}(\frac{c}{2})\right\}^{(14)}. \notag
\end{align}
By (2.22), we have
\begin{thm} One has the following equality:
\begin{align}
&\left\{\widehat{A}(TX){\rm exp}(\frac{c}{2})
{\rm ch}(\widetilde{T_CM}-\widetilde{L_C}+8+W_i+W_j)\right\}^{(14)}\\\notag
&=A\left\{e^{\frac{1}{24}A}\widehat{A}(TX){\rm exp}(\frac{c}{2})-\frac{e^{\frac{1}{24}A}-1}{A}
\widehat{A}(TX){\rm exp}(\frac{c}{2})
{\rm ch}(\widetilde{T_CM}-\widetilde{L_C}+8+W_i+W_j)\right\}^{(10)}
. \notag
\end{align}
\end{thm}
\begin{cor}
When $A=0$, we have
\begin{align}
&\left\{\widehat{A}(TX){\rm exp}(\frac{c}{2})
{\rm ch}(\widetilde{T_CM}-\widetilde{L_C}-16+W_i+W_j)\right\}^{(14)}\\\notag
&=-24\left\{\widehat{A}(TX){\rm exp}(\frac{c}{2})\right\}^{(14)}
. \notag
\end{align}
When $M$ is a $14$-dimensional $spin^c$ manifold, then ${\rm Ind}D^c\otimes(\widetilde{T_CM}-\widetilde{L_C}-16+W_i+W_j)_+$ is a multiply of $24$.
\end{cor}
Let
$A_1:=p_1(M)-p_1(L_R)+\frac{1}{30}c_2(W_i).$
Let \begin{align}
Q(M,P_i,\tau)=&\left\{e^{\frac{1}{24}E_2(\tau)A_1}\widehat{A}(TX){\rm exp}(\frac{c}{2}){\rm ch}\left[\bigotimes _{n=1}^{\infty}S_{q^n}(\widetilde{T_CM})
\otimes
\bigotimes _{m=1}^{\infty}\wedge_{-q^m}(\widetilde{L_C})\right]\right.\\\notag
&\left.\cdot\varphi(\tau)^{8}{\rm ch}(\mathcal{V}_i)\right\}^{(14)}.
\end{align}
Then
\begin{align}&Q(M,P_i,\tau)=\left\{e^{\frac{1}{24}E_2(\tau)A_1}\left(\prod_{j=1}^{7}\frac{x_j\theta'(0,\tau)}{\theta(x_j,\tau)}\right)
\frac{\sqrt{-1}\theta(u,\tau)}{\theta_1(0,\tau)\theta_2(0,\tau)
\theta_3(0,\tau)}\right.\\\notag
&\left.\frac{1}{2}\left(\prod_{l=1}^8\theta_1(y_l^i,\tau)+\prod_{l=1}^8\theta_2(y_l^i,\tau)+\prod_{l=1}^8\theta_3(y_l^i,\tau)\right)
\right\}^{(14)}.
\end{align}
Similar to Lemma 2.2, we have
\begin{lem}
$Q(M,P_i,\tau)$ is a modular form over $SL_2({\bf Z})$ with the weight $10$.
\end{lem}
Similar to (2.18), we have
\begin{align}
&e^{\frac{1}{24}E_2(\tau)A_1}\widehat{A}(TX){\rm exp}(\frac{c}{2}){\rm ch}\left[\bigotimes _{n=1}^{\infty}S_{q^n}(\widetilde{T_CM})
\otimes
\bigotimes _{m=1}^{\infty}\wedge_{-q^m}(\widetilde{L_C})\right]\varphi(\tau)^{8}{\rm ch}(\mathcal{V}_i)\\\notag
&=e^{\frac{1}{24}A_1}\widehat{A}(TX){\rm exp}(\frac{c}{2})+\left[e^{\frac{1}{24}A_1}\widehat{A}(TX){\rm exp}(\frac{c}{2})
{\rm ch}(\widetilde{T_CM}-\widetilde{L_C}-8+W_i)\right.\\\notag
&\left.-e^{\frac{1}{24}A_1}A_1\widehat{A}(TX){\rm exp}(\frac{c}{2})\right]q+
O(q^2).\notag
\end{align}
Since $Q(M,P_i,\tau)$ is a modular form over $SL_2({\bf Z})$ with the weight $10$, it must be
a multiple of
\begin{equation}
E_4(\tau)E_6(\tau)=1-264q-135432q^2+\cdots.
\end{equation}
So
\begin{align}
&\left\{e^{\frac{1}{24}A_1}\widehat{A}(TX){\rm exp}(\frac{c}{2})
{\rm ch}(\widetilde{T_CM}-\widetilde{L_C}-8+W_i)-e^{\frac{1}{24}A_1}A_1\widehat{A}(TX){\rm exp}(\frac{c}{2})\right\}^{(14)}\\\notag
&=-264\left\{e^{\frac{1}{24}A_1}\widehat{A}(TX){\rm exp}(\frac{c}{2})\right\}^{(14)}.
\end{align}
\begin{thm} One has the following equality:
\begin{align}
&\left\{\widehat{A}(TX){\rm exp}(\frac{c}{2})
{\rm ch}(\widetilde{T_CM}-\widetilde{L_C}+256+W_i)\right\}^{(14)}\\\notag
&=A_1\left\{e^{\frac{1}{24}A_1}\widehat{A}(TX){\rm exp}(\frac{c}{2})-\frac{e^{\frac{1}{24}A_1}-1}{A_1}
\widehat{A}(TX){\rm exp}(\frac{c}{2})
{\rm ch}(\widetilde{T_CM}-\widetilde{L_C}+256+W_i)\right\}^{(10)}
. \notag
\end{align}
\end{thm}
\begin{cor}
When $A_1=0$, we have
\begin{align}
&\left\{\widehat{A}(TX){\rm exp}(\frac{c}{2})
{\rm ch}(\widetilde{T_CM}-\widetilde{L_C}-8+W_i)\right\}^{(14)}\\\notag
&=-264\left\{\widehat{A}(TX){\rm exp}(\frac{c}{2})\right\}^{(14)}
. \notag
\end{align}
When $M$ is a $14$-dimensional $spin^c$ manifold, then ${\rm Ind}D^c\otimes(\widetilde{T_CM}-\widetilde{L_C}-8+W_i)_+$ is a multiply of $264$.
\end{cor}
If $M$ is a $10$-dimensional $spin^c$ manifold, we let

 \begin{align}
R(M,P_i,\tau)=&\left\{e^{\frac{1}{24}E_2(\tau)A_1}\widehat{A}(TX){\rm exp}(\frac{c}{2}){\rm ch}\left[\bigotimes _{n=1}^{\infty}S_{q^n}(\widetilde{T_CM})
\otimes
\bigotimes _{m=1}^{\infty}\wedge_{-q^m}(\widetilde{L_C})\right]\right.\\\notag
&\left.\cdot\varphi(\tau)^{8}{\rm ch}(\mathcal{V}_i)\right\}^{(10)}.
\end{align}
Then
\begin{align}&Q(M,P_i,\tau)=\left\{e^{\frac{1}{24}E_2(\tau)A_1}\left(\prod_{j=1}^{5}\frac{x_j\theta'(0,\tau)}{\theta(x_j,\tau)}\right)
\frac{\sqrt{-1}\theta(u,\tau)}{\theta_1(0,\tau)\theta_2(0,\tau)
\theta_3(0,\tau)}\right.\\\notag
&\left.\frac{1}{2}\left(\prod_{l=1}^8\theta_1(y_l^i,\tau)+\prod_{l=1}^8\theta_2(y_l^i,\tau)+\prod_{l=1}^8\theta_3(y_l^i,\tau)\right)
\right\}^{(10)}.
\end{align}
Similar to Lemma 2.2, we have
\begin{lem}
$R(M,P_i,\tau)$ is a modular form over $SL_2({\bf Z})$ with the weight $8$.
\end{lem}
By (2.27) and Lemma 2.8, it must be
a multiple of
\begin{equation}
E_4(\tau)^2=1+480q+61920q^2+\cdots.
\end{equation}
So
\begin{align}
&\left\{e^{\frac{1}{24}A_1}\widehat{A}(TX){\rm exp}(\frac{c}{2})
{\rm ch}(\widetilde{T_CM}-\widetilde{L_C}-8+W_i)-e^{\frac{1}{24}A_1}A_1\widehat{A}(TX){\rm exp}(\frac{c}{2})\right\}^{(10)}\\\notag
&=480\left\{e^{\frac{1}{24}A_1}\widehat{A}(TX){\rm exp}(\frac{c}{2})\right\}^{(10)}.
\end{align}
\begin{thm} One has the following equality:
\begin{align}
&\left\{\widehat{A}(TX){\rm exp}(\frac{c}{2})
{\rm ch}(\widetilde{T_CM}-\widetilde{L_C}-488+W_i)\right\}^{(10)}\\\notag
&=A_1\left\{e^{\frac{1}{24}A_1}\widehat{A}(TX){\rm exp}(\frac{c}{2})-\frac{e^{\frac{1}{24}A_1}-1}{A_1}
\widehat{A}(TX){\rm exp}(\frac{c}{2})
{\rm ch}(\widetilde{T_CM}-\widetilde{L_C}-488+W_i)\right\}^{(6)}
. \notag
\end{align}
\end{thm}
\begin{cor}
When $A_1=0$, we have
\begin{align}
&\left\{\widehat{A}(TX){\rm exp}(\frac{c}{2})
{\rm ch}(\widetilde{T_CM}-\widetilde{L_C}+W_i)\right\}^{(10)}=488\left\{\widehat{A}(TX){\rm exp}(\frac{c}{2})\right\}^{(10)}
. \notag
\end{align}
When $M$ is a $10$-dimensional $spin^c$ manifold, then ${\rm Ind}D^c\otimes(\widetilde{T_CM}-\widetilde{L_C}+W_i)_+$ is a multiply of $488$.
\end{cor}
In the following, we consider the $q^2$-terms in the $q$-expansions of $Q(M,P_i,P_j,\tau),~Q(M,P_i,\tau)$, $R(M,P_i,\tau).$ Let $\overline{W_i}=
P_i\times_{\rho_2}V_2$. We assume that $A=0$, then

\begin{align}
&e^{\frac{1}{24}E_2(\tau)A}\widehat{A}(TX){\rm exp}(\frac{c}{2}){\rm ch}\left[\bigotimes _{n=1}^{\infty}S_{q^n}(\widetilde{T_CM})
\otimes
\bigotimes _{m=1}^{\infty}\wedge_{-q^m}(\widetilde{L_C})\right]\varphi(\tau)^{16}{\rm ch}(\mathcal{V}_i){\rm ch}(\mathcal{V}_j)\\\notag
&=\widehat{A}(TX){\rm exp}(\frac{c}{2})
{\rm ch}\left[1+qB_1+q^2B_2+O(q^3)\right]
\\\notag
&\cdot(1-16q+104q^2))(1+{\rm ch}({W_i})q+q^2{\rm ch}(\overline{W_i})+O(q^3))(1+{\rm ch}({W_j})q+q^2{\rm ch}(\overline{W_j})+O(q^3))\\\notag
&=\widehat{A}(TX){\rm exp}(\frac{c}{2})+\left[\widehat{A}(TX){\rm exp}(\frac{c}{2})
{\rm ch}(\widetilde{T_CM}-\widetilde{L_C}-16+W_i+W_j)\right]q\\\notag
&+\left\{\widehat{A}(TX){\rm exp}(\frac{c}{2})({\rm ch}(\overline{W_j})+{\rm ch}(\overline{W_i})+
{\rm ch}({W_i}){\rm ch}({W_j})-16{\rm ch}({W_j})-16{\rm ch}({W_i})+104\right.\\\notag
&\left.+B_1\wedge[{\rm ch}({W_j})+{\rm ch}({W_i})-16]+B_2)\right\}q^2
+O(q^3).\notag
\end{align}
where
\begin{align}
&B_1=\widetilde{T_CM}-\widetilde{L_C},\\\notag
&B_2=\wedge^2\widetilde{L_C}-\widetilde{L_C}
-\widetilde{T_CM}\otimes \widetilde{L_C}+S^2\widetilde{T_CM}+\widetilde{T_CM}.\notag
\end{align}
By (2.21) and (2.37), we get
\begin{thm} When $A=0$, we have
\begin{align}
&\left\{\widehat{A}(TX){\rm exp}(\frac{c}{2})({\rm ch}(\overline{W_j})+{\rm ch}(\overline{W_i})+
{\rm ch}({W_i}){\rm ch}({W_j})-16{\rm ch}({W_j})-16{\rm ch}({W_i})+104\right.\\\notag
&\left.+B_1\wedge[{\rm ch}({W_j})+{\rm ch}({W_i})-16]+B_2)\right\}^{(14)}\\\notag
&=-196632\left\{\widehat{A}(TX){\rm exp}(\frac{c}{2})\right\}^{(14)}.\notag
\end{align}
\end{thm}

We assume that $A_1=0$, then
\begin{align}
&e^{\frac{1}{24}E_2(\tau)A_1}\widehat{A}(TX){\rm exp}(\frac{c}{2}){\rm ch}\left[\bigotimes _{n=1}^{\infty}S_{q^n}(\widetilde{T_CM})
\otimes
\bigotimes _{m=1}^{\infty}\wedge_{-q^m}(\widetilde{L_C})\right]\varphi(\tau)^{8}{\rm ch}(\mathcal{V}_i)\\\notag
&=\widehat{A}(TX){\rm exp}(\frac{c}{2})
{\rm ch}\left[1+qB_1+q^2B_2+O(q^3)\right]
\\\notag
&\cdot(1-8q+20q^2))(1+{\rm ch}({W_i})q+q^2{\rm ch}(\overline{W_i})+O(q^3)))\\\notag
&=\widehat{A}(TX){\rm exp}(\frac{c}{2})+\left[\widehat{A}(TX){\rm exp}(\frac{c}{2})
{\rm ch}(\widetilde{T_CM}-\widetilde{L_C}-8+W_i)\right]q\\\notag
&+\left\{\widehat{A}(TX){\rm exp}(\frac{c}{2})(20+{\rm ch}(\overline{W_i})
-8{\rm ch}({W_i})\right.\\\notag
&\left.+B_1\wedge[{\rm ch}({W_i})-8]+B_2)\right\}q^2
+O(q^3).\notag
\end{align}
By (2.28) and (2.40), we get
\begin{thm} When $A_1=0$, we have
\begin{align}
&\left\{\widehat{A}(TX){\rm exp}(\frac{c}{2})(20+{\rm ch}(\overline{W_i})
-8{\rm ch}({W_i})+B_1\wedge[{\rm ch}({W_i})-8]+B_2)\right\}^{(14)}\\\notag
&=-135432\left\{\widehat{A}(TX){\rm exp}(\frac{c}{2})\right\}^{(14)}.\notag
\end{align}
\end{thm}
Similarly, we have:
\begin{thm} When $A_1=0$, we have
\begin{align}
&\left\{\widehat{A}(TX){\rm exp}(\frac{c}{2})(20+{\rm ch}(\overline{W_i})
-8{\rm ch}({W_i})+B_1\wedge[{\rm ch}({W_i})-8]+B_2)\right\}^{(10)}\\\notag
&=61920\left\{\widehat{A}(TX){\rm exp}(\frac{c}{2})\right\}^{(10)}.\notag
\end{align}
\end{thm}

\section{Anomaly cancellation formulas for $12$-dimensional manifolds}
\indent In this section, we let $M$ be a $12$-dimensional manifold.
Let $\widehat{L}(TM,\nabla^{ TM})$
 be the Hirzebruch characteristic forms defined by (\cite{Zh})
 $$\widehat{L}(TM,\nabla^{ TM})={\rm
 det}^{\frac{1}{2}}\left(\frac{\frac{\sqrt{-1}}{2\pi}R^{TM}}{{\rm
 tanh}(\frac{\sqrt{-1}}{4\pi}R^{TM})}\right).$$
Let
$A_2:=\frac{1}{30}(c_2(W_i)+c_2(W_j)).$
Let \begin{align}
{Q_1}(M,P_i,P_j,\tau)=&\left\{e^{\frac{1}{24}E_2(\tau)A_2}\widehat{L}(TM,\nabla^{ TM}){\rm ch}\left[\bigotimes _{n=1}^{\infty}S_{q^n}(\widetilde{T_CM})
\otimes
\bigotimes _{m=1}^{\infty}\wedge_{q^m}(\widetilde{T_CM})\right]\right.\\\notag
&\left.\cdot\varphi(\tau)^{16}{\rm ch}(\mathcal{V}_i){\rm ch}(\mathcal{V}_j)\right\}^{(12)},
\end{align}
\begin{align}
{Q_2}(M,P_i,P_j,\tau)=&\left\{e^{\frac{1}{24}E_2(\tau)A_2}\widehat{A}(TM,\nabla^{ TM}){\rm ch}\left[\bigotimes _{n=1}^{\infty}S_{q^n}(\widetilde{T_CM})
\otimes
\bigotimes _{m=1}^{\infty}\wedge_{-q^{m-\frac{1}{2}}}(\widetilde{T_CM})\right]\right.\\\notag
&\left.\cdot\varphi(\tau)^{16}{\rm ch}(\mathcal{V}_i){\rm ch}(\mathcal{V}_j)\right\}^{(12)},
\end{align}
Then
\begin{align}&Q_1(M,P_i,P_j,\tau)=2^6\left\{e^{\frac{1}{24}E_2(\tau)A_2}\left(\prod_{j=1}^{6}\frac{x_j\theta'(0,\tau)\theta_1(x_j,\tau)}
{\theta(x_j,\tau)\theta_1(0,\tau)}\right)
\right.\\\notag
&\left.\frac{1}{4}\left(\prod_{l=1}^8\theta_1(y_l^i,\tau)+\prod_{l=1}^8\theta_2(y_l^i,\tau)+\prod_{l=1}^8\theta_3(y_l^i,\tau)\right)
\left(\prod_{l=1}^8\theta_1(y_l^j,\tau)+\prod_{l=1}^8\theta_2(y_l^j,\tau)+\prod_{l=1}^8\theta_3(y_l^j,\tau)\right)
\right\}^{(12)},
\end{align}
and
\begin{align}&Q_2(M,P_i,P_j,\tau)=\left\{e^{\frac{1}{24}E_2(\tau)A_2}\left(\prod_{j=1}^{6}\frac{x_j\theta'(0,\tau)\theta_2(x_j,\tau)}
{\theta(x_j,\tau)\theta_2(0,\tau)}\right)
\right.\\\notag
&\left.\frac{1}{4}\left(\prod_{l=1}^8\theta_1(y_l^i,\tau)+\prod_{l=1}^8\theta_2(y_l^i,\tau)+\prod_{l=1}^8\theta_3(y_l^i,\tau)\right)
\left(\prod_{l=1}^8\theta_1(y_l^j,\tau)+\prod_{l=1}^8\theta_2(y_l^j,\tau)+\prod_{l=1}^8\theta_3(y_l^j,\tau)\right)
\right\}^{(12)},
\end{align}
Let $$\Gamma_0(2)=\left\{\left(\begin{array}{cc}
\ a & b  \\
 c  & d
\end{array}\right)\in SL_2({\bf Z})\mid c\equiv 0~({\rm
mod}~2)\right\},$$
$$\Gamma^0(2)=\left\{\left(\begin{array}{cc}
\ a & b  \\
 c  & d
\end{array}\right)\in SL_2({\bf Z})\mid b\equiv 0~({\rm
mod}~2)\right\},$$  be the two modular subgroups of $SL_2({\bf Z})$.
Writing $\theta_j=\theta_j(0,\tau),~1\leq
j\leq 3,$ we introduce four explicit modular forms (\cite{Li1}),
\begin{align}
&\delta_1(\tau)=\frac{1}{8}(\theta_2^4+\theta_3^4),~~\varepsilon_1(\tau)=\frac{1}{16}\theta_2^4\theta_3^4,\\\notag
&\delta_2(\tau)=-\frac{1}{8}(\theta_1^4+\theta_3^4),~~\varepsilon_2(\tau)=\frac{1}{16}\theta_1^4\theta_3^4.\notag
\end{align}
\noindent They have the following Fourier expansions in
$q^{\frac{1}{2}}$: \begin{align}
&\delta_1(\tau)=\frac{1}{4}+6q+6q^2+\cdots,~~\varepsilon_1(\tau)=\frac{1}{16}-q+7q^2+\cdots,\\\notag
&8\delta_2(\tau)=-1-24q^{\frac{1}{2}}-24q-96q^{\frac{3}{2}}+\cdots,~~\varepsilon_2(\tau)=q^{\frac{1}{2}}+8q+28q^{\frac{3}{2}}+\cdots.\notag
\end{align}
 They also satisfy the
transformation laws,
\begin{align}\delta_2(-\frac{1}{\tau})=\tau^2\delta_1(\tau),~~~~~~\varepsilon_2(-\frac{1}{\tau})
=\tau^4\varepsilon_1(\tau),
\end{align}
\begin{lem}(\cite{Li1})$\delta_1(\tau)$ (resp.
$\varepsilon_1(\tau)$) is a modular form of weight $2$ (resp. $4$)
over $\Gamma_0(2)$, $\delta_2(\tau)$ (resp. $\varepsilon_2(\tau)$)
is a modular form of weight $2$ (resp. $4$) over $\Gamma^0(2)$ and
moreover ${\mathcal{M}}_{{\bf R}}(\Gamma^0(2))={\bf
R}[\delta_2(\tau),\varepsilon_2(\tau)]$.
\end{lem}
By (2.11)-(2.15) and (2.7), we have
\begin{lem}
${Q_1}(M,P_i,P_j,\tau)$ is a
modular form of weight $14$ over $\Gamma_0(2)$, while ${Q_2}(M,P_i,P_j,\tau)$ is
a modular form of weight $14$ over $\Gamma^0(2)$ . Moreover, the
following identity holds,
\begin{align}
{Q_1}(M,P_i,P_j,-\frac{1}{\tau})=2^{6}\tau^{14}{Q_2}(M,P_i,P_j,\tau).
\end{align}
\end{lem}
\begin{thm} We have the following equality:
\begin{align}
&32\left\{e^{\frac{1}{24}A_2}\widehat{L}(TM)\right\}^{(12)}=
\left\{e^{\frac{1}{24}A_2}\widehat{A}(TM){\rm ch}[2240+309\widetilde{T_CM}+24(\wedge^2\widetilde{T_CM}+W_i+W_j)\right.\\\notag
&\left.+\widetilde{T_CM}\otimes \widetilde{T_CM}+
\widetilde{T_CM}\otimes(W_i+W_j)+\wedge^3\widetilde{T_CM}]-
\frac{1}{24}e^{\frac{1}{24}A_2}A_2\widehat{A}(TX){\rm ch}(576+24\widetilde{T_CM})\right\}^{(12)}.
\notag
\end{align}
\end{thm}
\begin{proof}
By Lemmas 3.1 and 3.2, we have
\begin{align}
&{Q_2}(M,P_i,P_j,\tau)=h_0(8\delta_2)^{7}+h_1(8\delta_2)^{5}\varepsilon_2+h_2(8\delta_2)^{3}\varepsilon_2^2+h_3(8\delta_2)\varepsilon_2^3,
\end{align}
\begin{align}
&{Q_1}(M,P_i,P_j,\tau)=2^6[h_0(8\delta_1)^{7}+h_1(8\delta_1)^{5}\varepsilon_1+h_2(8\delta_1)^{3}\varepsilon_1^2+h_3(8\delta_1)\varepsilon_1^3],
\end{align}
where
each $h_r,~ 0\leq r\leq 3,$ is a real multiple of the
volume form at $x$.
\begin{align}
&e^{\frac{1}{24}E_2(\tau)A_2}\widehat{A}(TM,\nabla^{TM}){\rm ch}\left[\bigotimes _{n=1}^{\infty}S_{q^n}(\widetilde{T_CM})
\otimes
\bigotimes _{m=1}^{\infty}\wedge_{-q^{m-\frac{1}{2}}}(\widetilde{T_CM})\right]\cdot\varphi(\tau)^{16}{\rm ch}(\mathcal{V}_i){\rm ch}(\mathcal{V}_j)\\\notag
&=(e^{\frac{1}{24}A_2}-e^{\frac{1}{24}A_2}A_2q+O(q^2))\widehat{A}(TM)
{\rm ch}\left[(1+q\widetilde{T_CM}+O(q^2))\right.\\\notag
&\left.\otimes (1-q^{\frac{1}{2}}\widetilde{T_CM}+q\wedge^2\widetilde{T_CM}-q^{\frac{3}{2}}
\wedge^3\widetilde{T_CM}+O(q^2))\otimes (1-q^{\frac{3}{2}}
\widetilde{T_CM})\right]\\\notag
&\cdot(1-16q+O(q^2))(1+{\rm ch}({W_i})q+O(q^2))(1+{\rm ch}({W_j})q+O(q^2))\\\notag
&=e^{\frac{1}{24}A_2}\widehat{A}(TM)-q^{\frac{1}{2}}e^{\frac{1}{24}A_2}\widehat{A}(TX){\rm ch}(\widetilde{T_CM})\\\notag
&+q[e^{\frac{1}{24}A_2}\widehat{A}(TM){\rm ch}(\widetilde{T_CM}+\wedge^2\widetilde{T_CM}+W_i+W_j-16)-e^{\frac{1}{24}A_2}A_2\widehat{A}(TM)]\\\notag
&+q^{\frac{3}{2}}\left[-e^{\frac{1}{24}A_2}\widehat{A}(TX)
{\rm ch}(\widetilde{T_CM}\otimes \widetilde{T_CM}+\widetilde{T_CM}\otimes(W_i+W_j-16)+\widetilde{T_CM}+\wedge^3\widetilde{T_CM})\right.\\\notag
&\left.+e^{\frac{1}{24}A_2}A_2\widehat{A}(TM){\rm ch}(\widetilde{T_CM})\right]+
O(q^2).\notag
\end{align}
By (3.10), (3.12) and (3.6), comparing the the constant term, $q^{\frac{1}{2}},q, q^{\frac{3}{2}}$ terms in (3.10), we get
\begin{align}
&h_0=-\left\{e^{\frac{1}{24}A_2}\widehat{A}(TM)\right\}^{(12)}\\\notag
&h_1=\left\{e^{\frac{1}{24}A_2}\widehat{A}(TM){\rm ch}(\widetilde{T_CM}+168)\right\}^{(12)}\\\notag
&h_2=\left\{e^{\frac{1}{24}A_2}\widehat{A}(TM){\rm ch}(-9224-129\widetilde{T_CM}-\wedge^2\widetilde{T_CM}-W_i-W_j)
+e^{\frac{1}{24}A_2}A_2\widehat{A}(TM)\right\}^{(12)}\\\notag
&h_3=\left\{e^{\frac{1}{24}A_2}\widehat{A}(TM){\rm ch}(\widetilde{T_CM}\otimes \widetilde{T_CM}+\widetilde{T_CM}\otimes(W_i+W_j-16)
+\widetilde{T_CM}+\wedge^3\widetilde{T_CM}\right.\\\notag
&+508704-6868(\widetilde{T_CM}+168)
+88\times 9224+88\times 129\widetilde{T_CM}\\\notag
&\left.+88\wedge^2\widetilde{T_CM}+88W_i+88W_j)
-e^{\frac{1}{24}A_2}A_2\widehat{A}(TM){\rm ch}(\widetilde{T_CM}+88)\right\}^{(12)}.\notag
\notag
\end{align}
By (3.13), (3.11) and (3.6),  comparing the the constant term in (3.11), we get Theorem 3.3.
\end{proof}
\begin{cor} We have the following equality when $A_2=0$:
\begin{align}
&32\left\{\widehat{L}(TM)\right\}^{(12)}=
\left\{\widehat{A}(TM){\rm ch}[2240+309\widetilde{T_CM}+24(\wedge^2\widetilde{T_CM}+W_i+W_j)\right.\\\notag
&\left.+\widetilde{T_CM}\otimes \widetilde{T_CM}+
\widetilde{T_CM}\otimes(W_i+W_j)+\wedge^3\widetilde{T_CM}]\right\}^{(12)}.
\notag
\end{align}
When $M$ is spin, then the characteristic number of the right hand of (3.14) is a  multiple of $32$.
\end{cor}
Let
$A_3:=\frac{1}{30}c_2(W_i).$
Let \begin{align}
{Q_1}(M,P_i,\tau)=&\left\{e^{\frac{1}{24}E_2(\tau)A_3}\widehat{L}(TM,\nabla^{ TM}){\rm ch}\left[\bigotimes _{n=1}^{\infty}S_{q^n}(\widetilde{T_CM})
\otimes
\bigotimes _{m=1}^{\infty}\wedge_{q^m}(\widetilde{T_CM})\right]\right.\\\notag
&\left.\cdot\varphi(\tau)^{8}{\rm ch}(\mathcal{V}_i)\right\}^{(12)},
\end{align}
\begin{align}
{Q_2}(M,P_i,\tau)=&\left\{e^{\frac{1}{24}E_2(\tau)A_3}\widehat{A}(TM,\nabla^{ TM}){\rm ch}\left[\bigotimes _{n=1}^{\infty}S_{q^n}(\widetilde{T_CM})
\otimes
\bigotimes _{m=1}^{\infty}\wedge_{-q^{m-\frac{1}{2}}}(\widetilde{T_CM})\right]\right.\\\notag
&\left.\cdot\varphi(\tau)^{8}{\rm ch}(\mathcal{V}_i)\right\}^{(12)}.
\end{align}
Then
\begin{align}&Q_1(M,P_i,\tau)=2^6\left\{e^{\frac{1}{24}E_2(\tau)A_3}\left(\prod_{j=1}^{6}\frac{x_j\theta'(0,\tau)\theta_1(x_j,\tau)}
{\theta(x_j,\tau)\theta_1(0,\tau)}\right)
\right.\\\notag
&\left.\frac{1}{2}\left(\prod_{l=1}^8\theta_1(y_l^i,\tau)+\prod_{l=1}^8\theta_2(y_l^i,\tau)+\prod_{l=1}^8\theta_3(y_l^i,\tau)\right)
\right\}^{(12)},
\end{align}
and
\begin{align}&Q_2(M,P_i,\tau)=\left\{e^{\frac{1}{24}E_2(\tau)A_3}\left(\prod_{j=1}^{6}\frac{x_j\theta'(0,\tau)\theta_2(x_j,\tau)}
{\theta(x_j,\tau)\theta_2(0,\tau)}\right)
\right.\\\notag
&\left.\frac{1}{2}\left(\prod_{l=1}^8\theta_1(y_l^i,\tau)+\prod_{l=1}^8\theta_2(y_l^i,\tau)+\prod_{l=1}^8\theta_3(y_l^i,\tau)\right)
\right\}^{(12)}.
\end{align}
We have
\begin{lem}
${Q_1}(M,P_i,\tau)$ is a
modular form of weight $10$ over $\Gamma_0(2)$, while ${Q_2}(M,P_i,\tau)$ is
a modular form of weight $10$ over $\Gamma^0(2)$ . Moreover, the
following identity holds,
\begin{align}
{Q_1}(M,P_i,-\frac{1}{\tau})=2^{6}\tau^{10}{Q_2}(M,P_i,\tau).
\end{align}
\end{lem}
\begin{thm} We have the following equality:
\begin{align}
&\left\{e^{\frac{1}{24}A_3}\widehat{L}(TM)\right\}^{(12)}=
-\frac{1}{2}\left\{e^{\frac{1}{24}A_3}\widehat{A}(TM){\rm ch}[17\widetilde{T_CM}+\wedge^2\widetilde{T_CM}+W_i+128]\right\}^{(12)}\\\notag
&+\frac{1}{60}\left\{c_2(W_i)
e^{\frac{1}{24}A_3}\widehat{A}(TX)\right\}^{(12)}.
\notag
\end{align}
\end{thm}
\begin{proof}

By Lemmas 3.1 and 3.5, we have
\begin{align}
&{Q_2}(M,P_i,\tau)=h_0'(8\delta_2)^{5}+h_1'(8\delta_2)^{3}\varepsilon_2+h_2'(8\delta_2)\varepsilon_2^2,
\end{align}
\begin{align}
&{Q_1}(M,P_i,\tau)=2^6[h_0'(8\delta_1)^{5}+h_1'(8\delta_1)^{3}\varepsilon_1+h_2'(8\delta_1)\varepsilon_1^2],
\end{align}
where
each $h_r',~ 0\leq r\leq 2,$ is a real multiple of the
volume form at $x$.
\begin{align}
&e^{\frac{1}{24}E_2(\tau)A_3}\widehat{A}(TM,\nabla^{TM}){\rm ch}\left[\bigotimes _{n=1}^{\infty}S_{q^n}(\widetilde{T_CM})
\otimes
\bigotimes _{m=1}^{\infty}\wedge_{-q^{m-\frac{1}{2}}}(\widetilde{T_CM})\right]\cdot\varphi(\tau)^{8}{\rm ch}(\mathcal{V}_i)\\\notag
&=(e^{\frac{1}{24}A_3}-e^{\frac{1}{24}A_3}A_3q+O(q^2))\widehat{A}(TM)
{\rm ch}\left[(1+q\widetilde{T_CM}+O(q^2))\right.\\\notag
&\otimes (1-q^{\frac{1}{2}}\widetilde{T_CM}+q\wedge^2\widetilde{T_CM}+O(q^{\frac{3}{2}}))
\\\notag
&\cdot(1-8q+O(q^2))(1+{\rm ch}({W_i})q+O(q^2))\\\notag
&=e^{\frac{1}{24}A_3}\widehat{A}(TM)-q^{\frac{1}{2}}e^{\frac{1}{24}A_3}\widehat{A}(TM){\rm ch}(\widetilde{T_CM})\\\notag
&+q[e^{\frac{1}{24}A_3}\widehat{A}(TM){\rm ch}(\widetilde{T_CM}+\wedge^2\widetilde{T_CM}+W_i-8)-e^{\frac{1}{24}A_3}A_3\widehat{A}(TM)]
+O(q^{\frac{3}{2}})\notag
\end{align}
By (3.21), (3.23) and (3.6), comparing the the constant term, $q^{\frac{1}{2}},q$ terms in (3.21), we get
\begin{align}
&h'_0=-\left\{e^{\frac{1}{24}A_3}\widehat{A}(TM)\right\}^{(12)}\\\notag
&h'_1=\left\{e^{\frac{1}{24}A_3}\widehat{A}(TM){\rm ch}(\widetilde{T_CM}+120)\right\}^{(12)}\\\notag
&h'_2=\left\{-e^{\frac{1}{24}A_3}\widehat{A}(TM){\rm ch}(3712+81\widetilde{T_CM}+\wedge^2\widetilde{T_CM}+W_i)
+e^{\frac{1}{24}A_3}A_3\widehat{A}(TM)\right\}^{(12)}
\notag
\end{align}
By (3.22), (3.24) and (3.6),  comparing the the constant term in (3.22), we get Theorem 3.6.
\end{proof}

\begin{cor} We have the following equality when $c_2(W_i)=0$:
\begin{align}
&-2\left\{\widehat{L}(TM)\right\}^{(12)}=
\left\{\widehat{A}(TM){\rm ch}[17\widetilde{T_CM}+\wedge^2\widetilde{T_CM}+W_i+128]\right\}^{(12)}.
\notag
\end{align}
When $M$ is a spin manifold, then ${\rm Ind}(D\otimes(17\widetilde{T_CM}+\wedge^2\widetilde{T_CM}+W_i+128)_+)$ is a multiple of $2$.
\end{cor}

Comparing the coefficient of $q$ in (3.22), by (3.24), we have
\begin{thm} We have the following equality:
\begin{align}
&\left\{e^{\frac{1}{24}A_3}\widehat{L}(TM)[-A_3+2{\rm ch}(\widetilde{T_CM})-8+{\rm ch}(W_i)]
\right\}^{(12)}\\\notag
&=\left\{e^{\frac{1}{24}A_3}\widehat{A}(TM){\rm ch}[2116\widetilde{T_CM}+4\wedge^2\widetilde{T_CM}+4W_i-15872]\right\}^{(12)}\\\notag
&-\frac{2}{15}\left\{c_2(W_i)
e^{\frac{1}{24}A_3}\widehat{A}(TX)\right\}^{(12)}.
\notag
\end{align}
\end{thm}
\begin{cor} We have the following equality when $c_2(W_i)=0$:
\begin{align}
&\left\{\widehat{L}(TM)[2{\rm ch}(\widetilde{T_CM})-8+{\rm ch}(W_i)]
\right\}^{(12)}\\\notag
&=\left\{\widehat{A}(TM){\rm ch}[2116\widetilde{T_CM}+4\wedge^2\widetilde{T_CM}+4W_i-15872]\right\}^{(12)}.
\notag
\end{align}
\end{cor}
\vskip 0.5 true cm
\noindent {\bf Remark:} In \cite{CH1}, Chen and Han constructed some modular forms on odd dimensional manifolds. 
We can construct some new modular forms on odd dimensional manifolds by twist the Chen-Han modular forms by $E_8$ bundles.\\

\section{Acknowledgements}

 The author was supported in part by NSFC No.11771070. The first author is indebted to Prof. F. Han for very helpful discussions. The authors also thank the referee for his (or her) careful reading and helpful comments.

\vskip 1 true cm


\bigskip
\bigskip

 \indent{1.Yong Wang, Corresponding author, School of Mathematics and Statistics,
Northeast Normal University, Changchun Jilin, 130024, China }\\
\indent E-mail: {\it wangy581@nenu.edu.cn }\\
\indent {2.Yuchen Yang, School of Mathematics and Statistics , Northeast
Normal University, Changchun, Jilin 130024, China }\\
\indent E-mail: {\it  yangyc580@nenu.edu.cn }

\end{document}